\documentclass[leqno]{amsart}
\usepackage{amsmath,amsfonts,amsthm,amssymb,indentfirst,epic,url,centernot}

\newtheorem{theorem}{Theorem}
\newtheorem{lemma}[theorem]{Lemma}
\newtheorem{corollary}[theorem]{Corollary}
\newtheorem{proposition}[theorem]{Proposition}

\newtheorem{question}[theorem]{Question}

\newcommand{\Q}{\mathbb Q}
\newcommand{\Z}{\mathbb Z}
\renewcommand{\r}{\mathrm}
\DeclareMathOperator{\lt}{length}
\newcommand{\m}{\mathbf{m}}

\begin{document}

\begin{center}
\texttt{Comments, corrections, and related references welcomed,
as always!}\\[.5em]
{\TeX}ed \today
\vspace{2em}
\end{center}

\title{Commuting matrices, and modules over Artinian local rings}
\thanks{This preprint is readable online at
\url{http://math.berkeley.edu/~gbergman/papers/unpub}\,.
}

\subjclass[2010]{Primary: 13E10, 13E15, 15A27,
Secondary: 13C13, 15A30, 16D60, 16P20.}
\keywords{Commuting matrices, lengths of faithful modules over
commutative local Artinian rings.}

\author{George M. Bergman}
\address{University of California\\
Berkeley, CA 94720-3840, USA}
\email{gbergman@math.berkeley.edu}

\begin{abstract}
Gerstenhaber \cite{MG} proves that any commuting pair
of $n\times n$ matrices over a field $k$ generates a
$\!k\!$-algebra $A$ of $\!k\!$-dimension $\leq n.$
A well-known example shows that the corresponding statement for
$4$ matrices is false.
The question for $3$ matrices is open.

Gerstenhaber's result can be looked at as a statement about
the relationship between the length of a
$\!2\!$-generator finite-dimensional commutative $\!k\!$-algebra
$A,$ and the lengths of faithful $\!A\!$-modules.
Wadsworth \cite{W} generalizes this result to a larger class of
commutative rings than those generated by two elements over a field.
We recover his result, with a slightly improved argument.

We then explore some examples, raise further questions,
and make a bit of progress toward answering some of these.

An appendix gives some lemmas on generation and subdirect
decompositions of modules over not necessarily commutative Artinian
rings, generalizing a special case noted in the paper.
\end{abstract}

\maketitle

When I drafted this note, I thought the main result, Theorem~\ref{T.Rt},
was new; but having learned that it is not, I probably
won't submit this for publication
unless I find further strong results to add.
However, others may find interesting
the observations, partial results, and questions noted below,
and perhaps make some progress on them.

Noncommutative ring theorists might find the lemmas in the
appendix, \S\ref{S.gens.socs}, of interest.
(I would be interested to know whether they are known.)

\section{Wadsworth's generalization of Gerstenhaber's result}\label{S.2-gens}

To introduce Wadsworth's strengthening of
the result of Gerstenhaber quoted in the first sentence of the abstract,
note that a $\!k\!$-algebra of $n\times n$ matrices
generated by two commuting matrices can be viewed as an
action of the polynomial ring $k[s,\,t]$ on the vector space $k^n.$
A class of rings generalizing those of the form $k[s,\,t]$
are the rings $R[t]$ for $R$ a principal ideal domain.
The analog of an action of $k[s,\,t]$ on a finite-dimensional
$\!k\!$-vector space is then an $\!R[t]\!$-module of finite length.
A key tool in our study of such actions will be

\begin{corollary}[to the Cayley-Hamilton Theorem; cf.\ {\cite[Lemma~1]{W}}]\label{C.CH}
Let $R$ be a commutative ring, $M$ an $\!R\!$-module which can be
generated by $n<\infty$ elements, and $f$ an endomorphism of $M.$
Then $f^n$ is an $\!R\!$-linear
combination of $1_M,$ $f,$ $\dots,$ $f^{n-1}.$
\end{corollary}

\begin{proof}
Write $M$ as a homomorphic image of $R^n.$
Since $R^n$ is projective, we can lift $f$ to an endomorphism $g$
thereof.
Since $g$ satisfies its characteristic polynomial, $f$
satisfies the same polynomial.
\end{proof}

This shows that the unital subalgebra of $\r{End}_R(M)$ generated
by $f$ will be spanned over $R$ by the $f^i$ with
$0\leq i<n,$ but doesn't
say anything about the size of the contribution of each $f^i.$
In the next lemma, we obtain information of that sort by re-applying the
above corollary to various $\!i\!$-generator submodules of~$M.$

\begin{lemma}\label{I_i}
Let $R$ be a commutative ring, $M$ a finitely generated
$\!R\!$-module, and $f$ an endomorphism of $M.$
For each $i\geq 0,$ let $I_i$ be the ideal of $R$ generated
by all elements $u$ such that
$u\,M$ is contained in an $\!f\!$-invariant submodule of $M$ which is
$\!i\!$-generated as an $\!R\!$-module.
\textup{(}So $I_0\subseteq R$ is the annihilator of $M,$
and we have $I_0 \subseteq I_1 \subseteq \dots,$ with the ideals
becoming $R$ once we reach an $i$ such that $M$ is
$\!i\!$-generated.\textup{)}

Let $A$ be the unital $\!R\!$-algebra
of endomorphisms of $M$ generated by $f.$
For each $i\geq -1,$ let $A_i$ be the $\!R\!$-submodule of $A$
spanned by $\{1,\ f,\ f^2,\ \dots\,,\ f^i\}.$
Then\\[.2em]
\textup{(i)} For all $i\geq 0,$ the $\!R\!$-module $A_i/A_{i-1}$
is a homomorphic image of $R/I_i.$
Hence\\[.2em]
\textup{(ii)} If $M$ has finite length, and can be generated
by $d$ elements, then $A$ has length
$\leq \sum_{i=0}^{d-1}\,\lt(R/I_i)$ as an $\!R\!$-module.
\end{lemma}

\begin{proof}
Since $A_i=Rf^i+A_{i-1},$ (i) will follow if we show that
$I_i\,f^i\subseteq A_{i-1},$ which by definition of $I_i$ is
equivalent to showing that, for each $u\in R$ such that $u\,M$
is contained in an $\!i\!$-generated $\!f\!$-invariant
submodule $M'$ of $M,$ we have $u f^i\in A_{i-1}.$
Given such $u$ and $M',$ let $f'$ be the restriction of $f$ to $M'.$
By Corollary~\ref{C.CH},
$f'^i$ is an $\!R\!$-linear combination of the lower powers of $f';$ say
$\sum_{j<i} a_j f'^j.$
Hence the restriction
of $f$ to $u\,M\subseteq M'$ satisfies the same relation.
This says that $f^i u = \sum_{j<i} a_j f^j u,$ equivalently, that
$u f^i=\sum_{j<i} u a_j f^j,$ showing that $u f^i\in A_{i-1},$ as
required.

We deduce (ii) by summing over the steps of the chain
$\{0\}=A_{-1}\subseteq A_0\subseteq\dots\subseteq A_{d-1}=A.$
\end{proof}

We can now prove

\begin{theorem}[Wadsworth {\cite[Theorem~1]{W}}, after Gerstenhaber {\cite[Theorem~2, p.\,245]{MG}}]\label{T.Rt}
Let $M$ be a module of finite length over a commutative
principal ideal domain $R,$ let $f$ be an endomorphism of $M,$ and
let $A$ be the unital $\!R\!$-algebra of
$\!R\!$-module endomorphisms of $M$ generated by $f.$
Then
\begin{equation}\begin{minipage}[c]{35pc}\label{d.length}
$\lt_R(A)\ \leq\ \lt_R(M).$
\end{minipage}\end{equation}
\end{theorem}

\begin{proof}
Because $R$ is a commutative principal ideal domain, we can write $M$
as $R/(q_0)\oplus\dots\oplus R/(q_{d-1}),$ where
$q_{d-1}\,|\,q_{d-2}\,|\,\dots\,|\,q_0$
(\cite[Theorem III.7.7, p.\,151]{SL.Alg}, with terms relabeled).
Note that for $i=0,\dots,d-1,$ the element $q_i$ annihilates
the summands $R/(q_i),\dots,R/(q_{d-1})$ of the above decomposition.
Hence $q_i\,M$ is generated by $i$ elements, so in
the notation of the preceding lemma, $q_i\in I_i.$
(In fact, it is not hard to check that $I_i=(q_i).)$
In particular, $\lt_R(R/I_i)\leq\lt_R(R/(q_i)).$

By that lemma, the length of $A$ as an $\!R\!$-module
is $\leq\sum \lt_R(R/I_i),$ which by the above inequality
is $\leq\sum\lt_R(R/(q_i))=\lt_R(M).$
\end{proof}

We remark that in the above results,
the {\em length} of a module is really a proxy for its image in the
Grothendieck group of the category of finite-length modules.
I did not so formulate the results for simplicity of
wording, and to avoid excluding readers not
familiar with that viewpoint.
This has the drawback that when we want to pass from that
lemma back to Gerstenhaber's result on $\!k\!$-dimensions
of matrix algebras, the lengths do not directly
determine these dimensions, which the elements of the
Grothendieck group would.
(E.g., the $\!\mathbb{R}\!$-algebras
$\mathbb{R}[s]/(s)$ and $\mathbb{R}[s]/(s^2+1)$ both have length~$1$
as modules, but their $\!\mathbb{R}\!$-dimensions
are $1$ and $2$ respectively.)

However, we can get around this by applying Theorem~\ref{T.Rt}
after extension of scalars to the algebraic closure of~$k.$
Indeed, if $A$ is a $\!k\!$-algebra of endomorphisms of
a finite-dimensional $\!k\!$-vector-space $M,$ then the
$\!k\!$-dimensions of $A$ and $M$ are
unaffected by extending scalars to the algebraic closure $\bar{k}$
of $k;$ and for $k$ algebraically closed, the length of $M$ as an
$\!A\!$-module is just its $\!k\!$-dimension.
Hence an application of Theorem~\ref{T.Rt} with $R=\bar{k}[s]$
to these extended modules and algebras gives us Gerstenhaber's
result for the original modules and algebras.

\section{Some notes on the literature}\label{S.digr}
The hard part of Gerstenhaber's proof of his result was a
demonstration that the variety (in the sense of algebraic
geometry) of all pairs of commuting $n\times n$ matrices is irreducible.
Guralnick~\cite{G} notes that this fact had been proved earlier
by Motzkin and Taussky~\cite{M+T}, and gives a shorter proof
of his own.

The first proof of Gerstenhaber's result by non-algebraic-geometric
methods is due to Barr\'{i}a and Halmos~\cite{B+H}.
Wadsworth \cite{W} abstracts that proof by replacing $k[s]$ by a
general principal ideal domain~$R.$
His argument differs from ours in that he
obtains Corollary~\ref{C.CH} only for
$R$ a principal ideal domain and $M$ a torsion $\!R\!$-module.
This restriction arises from his calling on
the fact that every such module $M$ {\em embeds} in
a free module over some factor-ring of $R$ (in the
notation of our proof of Theorem~\ref{T.Rt},
the free module of rank $d$ over $R/(q_0)),$ where we
use, instead, the fact that any finitely generated $\!R\!$-module
is a {\em homomorphic image} of a free $\!R\!$-module of finite rank,
which is true for any commutative ring $R.$
Of course, the restriction assumed in Wadsworth's
proof holds in the case to which we both apply the result;
but the more general statement of Corollary~\ref{C.CH} seemed
worth recording.

I mentioned above that where I speak of the length of a module, a more
informative statement would refer to its image in a Grothendieck group.
Wadsworth uses an invariant of finite-length
modules over PIDs that is equivalent to that more precise
information: in the notation of our proof of Theorem~\ref{T.Rt},
the equivalence class, under associates, of the
product $q_0\ldots q_{d-1}\in R.$
He also shows \cite[Theorem~2]{W} that as an $\!R\!$-module,
$A$ can in fact be embedded in $M.$
(However, we will note at the end of \S\ref{S.3-gen} that it
cannot in general be so embedded as an $\!R[t]\!$-module.)

Gerstenhaber~\cite{MG} proves a bit more about the algebra
generated by two commuting $n\times n$ matrices
than we have quoted: he also shows that it
is {\em contained} in a commutative matrix algebra
of dimension {\em exactly} $n,$ at least after
a possible extension of the base field.
(He mentions that he does not know whether this is true
without extension of the base field.)
We shall not discuss that property further here.

Guralnick and others \cite{G+S}, \cite{S+Sh} have continued
the algebraic geometric investigation of these questions.
Some investigations which, like this note, focus more on methods of
linear algebra are \cite{L+L}, \cite{H+O}, \cite{N+S}.
For an extensive study of the subject, see
O'Meara, Clark and Vinsonhaler \cite[Chapter~5]{O+C+V}.
\vspace{.3em}

Returning to Theorem~\ref{T.Rt}, the hypothesis
that $R$ be a principal ideal domain can be weakened,
with a little more work, to say that $R$ is
a Dedekind domain, or even a Pr\"{u}fer domain, since under
these assumptions, every finite-length homomorphic image of $R$ is
a direct product of uniserial rings, which is what is really needed
to get the indicated description of finite-length modules
(though I don't know a reference stating this description in
those cases).
We shall discuss in \S\S\ref{S.genlz}-\ref{S.other_Qs} wider
generalizations that one may hope for, and will make some
progress in those directions.

But for the next two sections, let us return to commuting
matrices over a field, and examine what can happen in
algebras generated by more than two such matrices.

\section{Counterexamples with $4$ generators}\label{S.4-gen}

The standard example showing that a $\!4\!$-generator algebra
of $n\times n$ matrices can have dimension $>n$ takes $n=4,$ and
for $A,$ the algebra of $4\times 4$ matrices
generated by $e_{13},$ $e_{14},$ $e_{23}$ and $e_{24}.$
These commute, since their pairwise products are
zero, and $A$ has for a basis these four elements and the identity
matrix $1,$ and so has dimension $5>4=n.$

If the reader finds it disappointing that the extra dimension comes
from the convention that algebras are unital, note
that without that convention, one can obtain the same subalgebra
from the four generators $1+e_{13},$ $e_{14},$ $e_{23},$ $e_{24},$
using the fact that the not-necessarily-unital algebra
generated by an upper triangular matrix (here $1+e_{13})$
always contains both the diagonal part (here $1)$ and the
above-diagonal part (here $e_{13})$ of that matrix.

One can modify this example to get
commutative $\!4\!$-generator matrix algebras
in which the dimension of the algebra exceeds the
size of the matrices by an arbitrarily large amount.
Namely, for any $m,$ let us form within the algebra of $4m\times 4m$
matrices a ``union of $m$ copies'' of each of the matrix units used
in the above example.
To do this, let
$E_{13}=\sum_{j=0}^{m-1} e_{4j+1,\,4j+3},$
$E_{14}=\sum_{j=0}^{m-1} e_{4j+1,\,4j+4},$
$E_{23}=\sum_{j=0}^{m-1} e_{4j+2,\,4j+3},$
$E_{24}=\sum_{j=0}^{m-1} e_{4j+2,\,4j+4},$ and let us also choose a
diagonal matrix $D$ having one value, $\alpha_0\in k-\{0\},$ in the
first four diagonal positions, a different
value, $\alpha_1\in k-\{0\},$ in the next four, and so on; in other
words, $D=\sum_{j=0}^{m-1}\,\alpha_j(\sum_{i=1}^4 e_{4j+i,\,4j+i}).$
(Here we assume $|k|\geq m+1,$ so that the $\alpha_i$ can be taken
distinct.)
We then take as our four generators
$D+E_{13},$ $E_{14},$ $E_{23}$ and $E_{24}.$
From the fact about upper triangular matrices called on in the
preceding paragraph, the algebra these generate
contains $D,$ $E_{13},$ $E_{14},$ $E_{23}$ and $E_{24}.$
Using just $D$ and the $\!k\!$-algebra structure, one gets the
$m$ diagonal idempotent elements $\sum_{i=1}^4 e_{4j+i,\,4j+i}$
$(j=0,\dots,m-1),$ and with the help of these, one sees that
our $\!4\!$-generator algebra is the
direct product of $m$ copies of the algebra of the preceding paragraphs.
Thus, we have a $\!4\!$-generated commutative algebra
of dimension $5m$ within the ring of $4m\times 4m$ matrices.

Because of the strong way this construction used a diagonal matrix
$D,$ I briefly hoped that if we restricted attention to
algebras $A$ generated by four commuting \emph{nilpotent} matrices,
the dimension of $A$ might never exceed the size of the
matrices by more than $1.$
But the following family of examples contradicts that guess.

In describing it, I will again use the language of
vector spaces and their endomorphisms.
Let $m$ be any positive integer, and let $M$ be a
$\!(5m^2+3m)/2\!$-dimensional $\!k\!$-vector-space, with basis
consisting of elements which we name (proactively)
\begin{equation}\begin{minipage}[c]{35pc}\label{d.abxy}
$a^i\,b^j\,x,$\ \ for $i,j\geq 0,\ i\,{+}\,j\leq 2m-1,$\qquad and\qquad
$a^i\,b^j\,y,$\ \ for $i,j\geq 0,\ i\,{+}\,j\leq m-1.$
\end{minipage}\end{equation}
(So we have $2m(2m+1)/2 = 2m^2+m$ basis elements $a^i\,b^j\,x,$ and
$m(m+1)/2 = (m^2+m)/2$ basis elements $a^i\,b^j\,y,$
totaling $(5m^2+3m)/2$ elements.)

We now define four linear maps, $a,$ $b,$ $c$ and $d$ on $M.$
Of these, $a$ and $b$ act in the obvious ways
on the elements $a^i\,b^j\,x$
with $i+j<2m-1$ and on the elements $a^i\,b^j\,y$ with $i+j<m-1,$
namely, by increasing the formal exponent of $a,$
respectively, $b,$ by $1;$ while they annihilate the elements
for which these formal exponents have their maximum allowed total value,
$i+j=2m-1,$ respectively $i+j=m-1.$
On the other hand, we let $c$ annihilate $x,$ but take $y$ to $a^m\,x,$
and hence (as it must if it is to commute with $a$ and $b)$
take $a^i\,b^j\,y$ to $a^{m+i}\,b^j\,x,$ and we similarly let
$d$ annihilate $x,$ but take $y$ to $b^m\,x,$
and hence $a^i\,b^j\,y$ to $a^i\,b^{m+j}\,x.$

It is easy to verify that these four linear maps commute,
and are nilpotent.
I claim that the unital algebra that they generate has for a basis
the elements
\begin{equation}\begin{minipage}[c]{35pc}\label{d.abcd}
$a^i\,b^j$ for $i,j\geq 0,\ i+j\leq 2m-1,$\qquad and\qquad
$a^i\,b^j\,c$ and $a^i\,b^j\,d$ for $i,j\geq 0,\ i+j\leq m-1$
\end{minipage}\end{equation}
(compare with~(\ref{d.abxy})).
Indeed, it is immediate that every monomial in $a,$ $b,$ $c$ and $d$
other than those listed in~(\ref{d.abcd}) is zero.
Now suppose some $\!k\!$-linear combination of the
monomials~(\ref{d.abcd}) were zero.
By applying that linear combination to $x\in M,$ we
see that the coefficients in $k$ of the monomials $a^i\,b^j$ (with no
factor $c$ or $d)$ are all zero.
Applying the same element to $y,$ and noting that the
sets of basis elements $a^{m+i}\,b^j\,x$ $(i+j\leq m-1)$
and $a^i\,b^{m+j}\,x$ $(i+j\leq m-1)$ are disjoint,
we conclude that the coefficients of
the monomials $a^i\,b^j\,c$ and $a^i\,b^j\,d$ are also zero.

Counting the elements~(\ref{d.abcd}), we see that
the dimension of our algebra is
$(2m^2+m)+(m^2+m)/2+(m^2+m)/2=3m^2+2m,$
and this exceeds that of $M$ by $(m^2+m)/2,$ which is unbounded
as $m$ grows.
The limit as $m\to\infty$ of the ratio of the dimensions of $A$ and $M$
is $\r{lim}\,(3m^2+2m)/((5m^2+3m)/2)=6/5.$

Can we get a similar example with limiting ratio $5/4,$ the
ratio occurring for our $\!4\!$-dimensional $M$?
In fact we can.
The description is formally like that of the above example, but with
the basis~(\ref{d.abxy}) replaced by the slightly more complicated
basis,
\begin{equation}\begin{minipage}[c]{35pc}\label{d.abxy_alt}
$a^i\,b^j\,x,$\ \ where $0\leq i,j<2m$ and $\min(i,j)<m,$\qquad
and\qquad $a^i\,b^j\,y,$\ \ where $0\leq i,j<m.$
\end{minipage}\end{equation}
We again define $a$ and $b$ to act by adding $1$ to the relevant
exponent symbol when this leads to another element of the above basis,
while taking basis elements to zero when it does not;
and we again let $c$ and $d$ annihilate $x,$ but carry
$y$ to $a^m\,x,$ respectively, $b^m\,x,$ and act on other
basis elements as they must for our operators to commute.
We find that $\dim_k\,M=4m^2$ while $\dim_k\,A=5m^2,$
so indeed, $\r{dim}_k\,A/\r{dim}_k\,M=5/4.$

Incidentally, the $m=1$ case of both the construction
using~(\ref{d.abxy}) and the one using~(\ref{d.abxy_alt})
can be seen to be isomorphic to the $\!5\!$-dimensional
algebra of $4\times 4$ matrices with which we began this section.

\section{The recalcitrant $\!3\!$-generator question}\label{S.3-gen}

It is not known whether every $\!3\!$-generator
commutative $\!k\!$-algebra $A$ of endomorphisms of a finite-dimensional
$\!k\!$-vector-space $V$ satisfies $\r{dim}_k\,A\leq \r{dim}_k\,V.$
Let me lead into the discussion of that question by
starting with some observations applicable to any
commutative algebra $A$ of endomorphisms of a finite-dimensional vector
space $V.$
We will again write $M$ for $V$ regarded as an $\!A\!$-module.
Note that since $A$ is an algebra of endomorphisms
of $M,$ it acts faithfully on $M.$

Since $A$ is a finite-dimensional commutative $\!k\!$-algebra, it is
a direct product of local algebras, and the idempotents
arising from the decomposition of $1\in A$ yield a corresponding
decomposition of $M$ as a direct product of modules over
one or another of these factors.
Thus the question of whether $\r{dim}_k\,A$ can exceed $\r{dim}_k\,M$
reduces, in general, to the corresponding question for local algebras;
so we shall assume $A$ local in what follows.
Moreover, since passing to the algebraic closure of $k$ does
not affect the properties we are interested in, we can assume that
$k$ is algebraically closed, hence that the residue
field of the local algebra $A$ is $k$ itself.
(Having reduced our considerations to this case, we can now
drop the hypothesis that $k$ be algebraically closed, keeping
only this condition on the residue field.)
Thus, the $\!k\!$-algebra structure of $A$ is determined by that of
its maximal ideal, which, by finite-dimensionality, is nilpotent.
We shall denote this ideal $\m.$

Next note that if $M$ is cyclic as an $\!A\!$-module,
then it will be isomorphic as an $\!A\!$-module to $A$ itself.
(This is a consequence of commutativity.
If $M=A\,x,$ and some nonzero element of $A$ annihilated $x,$
then by commutativity, it would annihilate all of $M,$ contradicting
the assumption that $A$ acts faithfully.)

Note also that the dual space $M^*=\r{Hom}_k(M,\,k)$ acquires
a natural structure of $\!A\!$-module (since the vector-space
endomorphisms of $M$ given by the elements of $A$ induce
endomorphisms of $M^*),$ of the same $\!k\!$-dimension as $M.$
Using this duality, it follows from the preceding observations
that $\r{dim}_k\,A=\r{dim}_k\,M$ will also hold if $M^*$ is cyclic.
The latter condition is equivalent to saying that
the socle of $M$ (the annihilator of $\m$ in $M)$
is $\!1\!$-dimensional;
in this situation one calls $M$ ``cocyclic''.
(Dually, the condition that $M$ be cyclic is equivalent to
saying that $\m M$ has codimension $1$ in $M.)$

Thus, if $\r{dim}_k\,A$ and $\r{dim}_k\,M$ are to be distinct,
$M$ can be neither cyclic nor cocyclic.
For brevity of exposition, let us focus on the consequences of
$M$ being noncyclic.

The most simpleminded way to get a noncyclic module $M$ is to
take a direct sum of two nonzero cyclic modules.
(In doing so, we keep the assumption that $A$ acts faithfully on $M,$
though it will not in general act faithfully on these summands.)
In this situation, I claim that $A$ will have dimension {\em strictly
smaller} than that of $M.$
Indeed, writing the direct summands as $M_1\cong A/I_1$ and
$M_2\cong A/I_2$ for proper ideals $I_1,$ $I_2\subseteq A,$
we see that the $\!k\!$-dimension of $M$ is the sum
of the codimensions of $I_1$ and $I_2$ in $A,$ while the
dimension of $A$ is the codimension of the zero ideal,
which in this case is $I_1\cap I_2.$
In view of the exact sequence
$0\to A/(I_1\cap I_2)\to A/I_1\oplus A/I_2\to A/(I_1+I_2)\to 0,$
the dimension of the direct sum $M$ must exceed that of the
algebra $A$ by $\r{dim}_k\,A/(I_1+I_2),$ which is positive
because $I_1$ and $I_2$ lie in the common ideal $\m.$

Hence, if we want a module $M$ such that
\begin{equation}\begin{minipage}[c]{35pc}\label{d.A>M}
$\r{dim}_k\,A\ >\ \r{dim}_k\,M,$
\end{minipage}\end{equation}
with $M$ generated by two elements $x_1$ and $x_2,$ so that it is
a homomorphic image of $A\,x_1\oplus A\,x_2,$
then the construction of this homomorphic image must involve
additional relations, i.e., the identification of a nonzero
submodule of $A\,x_1$ with an isomorphic submodule of~$A\, x_2.$

It turns out that a single relation, i.e., the identification
of a {\em cyclic} submodule of $A\,x_1$ with an
isomorphic cyclic submodule of $A\,x_2,$ is still not enough
to get~(\ref{d.A>M}).
For let the common isomorphism class of
the cyclic submodules that we identify be that of $A/I_3.$
Then $I_1$ and $I_2$ are both contained in $I_3,$ hence so
is $I_1+I_2.$
Now we have seen that the amount by which
$\r{dim}_k\,(A\,x_1\oplus A\,x_2)$ exceeds $\r{dim}_k\,A$ is
the codimension of $I_1+I_2$ in $A;$ so setting to zero a
submodule isomorphic to $A/I_3,$ which has dimension at most
that codimension, can at best give us equality.

Thus, we need to divide out by at least a $\!2\!$-generator
submodule to get~(\ref{d.A>M}).
And indeed, the families of $\!4\!$-generator examples we obtained
in the latter half of the preceding section can be thought of as
constructed by imposing two relations, $cy=a^mx$ and $dy=b^mx,$
on a direct sum $A\,x\oplus A\,y\cong A/I_1\oplus A/I_2.$

In the remainder of this section,
I will display a few examples diagrammatically.
In these examples, $M$ will have a $\!k\!$-basis $B$ such
that each of our given generators of $A$ carries each element of $B$
either to another element of $B$ or to $0.$
The actions of the various generators of $A$ on basis elements
will be shown as downward line segments of different slopes, with
the matching of generator and slope shown to the right
of the diagram (under the word ``labeling'').
Where no line segment of a given slope descends from a given
element, this means that the corresponding
generator of $A$ annihilates that basis element.
For instance, the $\!4\!$-dimensional example with
which we began the preceding section may be diagrammed

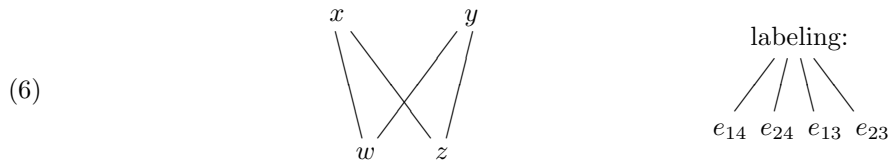
\begin{equation}\begin{minipage}[c]{20pc}\label{d.4-gen}
\begin{picture}(200,65)
\put(0,5){
\put(5,54){$x$}
    \put(9,50){\line(1,-4){10}}
    \put(15,50){\line(3,-4){30}}
\put(56,54){$y$}
    \put(55,50){\line(-3,-4){30}}
    \put(61,50){\line(-1,-4){10}}
\put(15,2){$w$}
\put(45,2){$z$}
\put(150,0){
\put(16,45){labeling:}
\put(25,40){\line(-3,-4){15}}
    \put(00,11){$e_{14}$}
\put(30,40){\line(-1,-4){5}}
    \put(18,11){$e_{24}$}
\put(35,40){\line(1,-4){5}}
    \put(36,11){$e_{13}$}
\put(40,40){\line(3,-4){15}}
    \put(54,11){$e_{23}$}
}
}
\end{picture}
\end{minipage}\end{equation}
The fact that it can be obtained from a direct sum of two cyclic
modules $A\,x_1,$ $A\,x_2$ by two identifications is made clear in the
representation below,
where the labels on the lower vertices show the relations imposed.

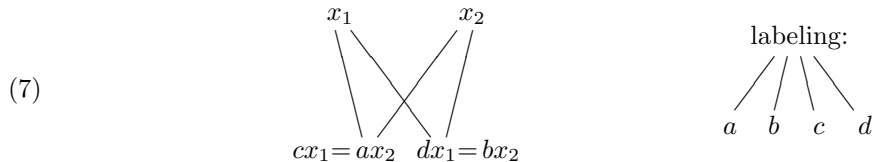
\begin{equation}\begin{minipage}[c]{20pc}\label{d.4-gen2}
\begin{picture}(200,65)
\put(0,5){
\put(4,54){$x_1$}
    \put(9,50){\line(1,-4){10}}
    \put(15,50){\line(3,-4){30}}
\put(54,54){$x_2$}
    \put(55,50){\line(-3,-4){30}}
    \put(61,50){\line(-1,-4){10}}
\put(-9,2){$cx_1{=}\,ax_2$}
\put(38,2){$dx_1{=}\,bx_2$}
\put(150,0){
\put(16,45){labeling:}
\put(25,40){\line(-3,-4){15}}
    \put(04,11){$a$}
\put(30,40){\line(-1,-4){5}}
    \put(21,11){$b$}
\put(35,40){\line(1,-4){5}}
    \put(38,11){$c$}
\put(40,40){\line(3,-4){15}}
    \put(55,11){$d$}
}
}
\end{picture}
\end{minipage}\end{equation}
(In~\eqref{d.4-gen2}, I have changed the edge-labeling
from~\eqref{d.4-gen}, since that labeling, based on the use of
upper triangular matrices in the preceding section, would
have led to lower-indexed basis elements on the bottom and
higher-indexed basis elements on the top, and hence,
when the former were replaced by expressions in the latter,
a notation where everything was expressed in terms
of elements apparently arbitrarily called $x_3$ and $x_4.)$

If we drop one of the generators of the above
$A,$ getting a $\!3\!$-generator algebra,
then the same vector space becomes,
as shown in the next diagram, a module whose
submodules $A\,x_1$ and $A\,x_2$ are connected by only one relation,
and which has $\r{dim}_k\,A=\r{dim}_k\,M$
(both dimensions being $4),$ which, we have
seen, is the best one can hope for from such a ``one-relation'' module.

\begin{equation}\begin{minipage}[c]{20pc}\label{d.3-gen1}
\begin{picture}(200,65)
\put(0,5){
\put(4,54){$x_1$}
    \put(9,50){\line(1,-4){10}}
\put(54,54){$x_2$}
    \put(55,50){\line(-3,-4){30}}
    \put(61,50){\line(-1,-4){10}}
\put(-9,2){$cx_1{=}\,ax_2$}
\put(43,2){$bx_2$}
\put(150,0){
\put(13,45){labeling:}
\put(25,40){\line(-3,-4){15}}
    \put(04,11){$a$}
\put(30,40){\line(-1,-4){5}}
    \put(21,11){$b$}
\put(35,40){\line(1,-4){5}}
    \put(38,11){$c$}
}
}
\end{picture}
\end{minipage}\end{equation}

We remark that we still get equality of dimensions if, in the
above example, we replace one or more of the $\!1\!$-step paths by
multistep paths repeating the algebra generator in question; e.g.,

\begin{equation}\begin{minipage}[c]{20pc}\label{d.3-gen2}
\begin{picture}(200,65)
\put(0,5){
\put(3,54){$x_1$}
   \put(9,50){\line(1,-4){10}
   \put(-6.7,-13.2){\circle*{1.7}}
   \put(-3.3,-26.4){\circle*{1.7}}}
\put(54,54){$x_2$}
   \put(61,50){\line(-1,-4){10}
   \put(5.0,-10){\circle*{1.7}}
   \put(2.5,-20){\circle*{1.7}}
   \put(0.0,-30){\circle*{1.7}}}
   \put(55,50){\line(-3,-4){30}
   \put(7.5,-20){\circle*{1.7}}}
\put(-14,0){$c^3x_1{=}\,a^2x_2$}
\put(43,0){$b^4x_2$}
\put(150,0){
\put(13,45){labeling:}
\put(25,40){\line(-3,-4){15}}
    \put(04,11){$a$}
\put(30,40){\line(-1,-4){5}}
    \put(21,11){$b$}
\put(35,40){\line(1,-4){5}}
    \put(38,11){$c$}
}
}
\end{picture}
\end{minipage}\end{equation}
which has $\r{dim}_k\,A=\r{dim}_k\,M=10.$

I have attempted to find examples of $\!3\!$-generator
algebras $A$ such that $\r{dim}_k\,A>\r{dim}_k\,M,$ by connecting
two or more
cyclic modules with the help of two or more relations, but without
success; the best my fiddling with such examples has achieved
is to get more examples of equality; for instance,

\begin{equation}\begin{minipage}[c]{20pc}\label{d.3-gen<>}
\begin{picture}(200,70)
\put(00,05){
\put(30,60){$x_1$}
\put(60,60){$x_2$}
\put(00,30){$y_0$}
\put(30,30){$y_1$}
\put(60,30){$y_2$}
\put(90,30){$y_0$}
\put(30,00){$z_1$}
\put(60,00){$z_2$}
\multiput(30,56)(30,0){2}{\multiput(0,0)(30,-30){2}{\line(-1,-1){19}}}
\multiput(40,56)(30,0){2}{\multiput(0,0)(-30,-30){2}{\line(1,-1){19}}}
\multiput(35,56)(30,0){2}{\multiput(0,0)(0,-30){2}{\line(0,-1){19}}}
}
\put(160,0){
\put(13,55){labeling:}
\put(20,50){\line(-1,-1){15}}
    \put(-1,25){$a$}
\put(28,50){\line(0,-4){15}}
    \put(25,25){$b$}
\put(36,50){\line(1,-1){15}}
    \put(50,25){$c$}
}
\end{picture}
\end{minipage}\end{equation}
Note the repetition of $y_0$ at the right end of the middle row;
thus, $y_0$ is both $a\,x_1$ and $c\,x_2.$
We find that the distinct nonzero monomials in $A$ are
\begin{equation}\begin{minipage}[c]{25pc}\label{d.dim7}
$1;$ $a,$ $b,$ $c;$ $a^2=bc,$ $b^2=ac,$ $c^2=ab,$
\end{minipage}\end{equation}
which are linearly independent,
so that $A,$ like $M,$ is $\!7\!$-dimensional.
(One must also verify commutativity of $A.$
This is fairly easy;
there are three relations to be checked, $ab=ba,$
$ac=ca$ and $bc=cb,$ and these need only be checked on
$x_1$ and $x_2,$ since on all other basis elements, both sides
of each equation clearly give $0.)$

One can modify this example by making the rightmost basis-element
in the middle row be, not a repetition of $y_0=ax_1,$ but
$a^iy_0=a^{i+1}x_1,$ for any $i>0.$
This adds exactly $i$ basis elements to $M,$
namely $a\,y_0,\dots,a^iy_0,$ and likewise adds $i$
monomials to the basis of $A,$ namely $a^3,\dots,a^{i+2}$ (with
$a^{i+2}$ rather than $a^2$ now coinciding with $bc).$
In the new $A,$ the relation $c^2=ab$ no longer holds;
rather, $c^2=0,$ but $ab$ remains nonzero.

We see that this modified example still satisfies
$\r{dim}_k A=\r{dim}_k M.$
In fact, Kevin O'Meara has pointed out to me that by
\cite[Theorem~5.5.8]{O+C+V},
if a $\!k\!$-algebra $A$ of endomorphisms of
a $\!k\!$-vector-space $M$ generated by three commuting
elements $a,$ $b,$ $c$ is to satisfy $\r{dim}_k\,A>\r{dim}_k\,M,$ then
$M$ must require at least $4$ generators as a module over the
subring $k[a].$
(The wording of that theorem is that $\r{dim}_k\,A\leq\r{dim}_k\,M$
holds for ``$\!3\!$-regular'' matrix algebras, i.e., those
where $M$ can be so generated by $3$ elements.)
The example of~(\ref{d.3-gen<>}), and the variant
just noted, are generated over $k[a]$ by $\{x_1,\,x_2,\,y_2\}$
confirming that we would need something more
complicated to get a counterexample.
O'Meara suspects that the analog of the theorem just quoted
also holds for
$\!4\!$-regular algebras, but might fail in the $\!5\!$-regular case.
(Cf.~\cite[p.226, footnote~12]{O+C+V}.)
Incidentally,~(\ref{d.3-gen2}) is an example
of a module that is not $\!3\!$-generated over any of
$k[a],$ $k[b],$ $k[c],$ but which still
satisfies $\r{dim}_k A=\r{dim}_k M.$

We remark, for the benefit of the reader who wants to
explore examples using diagrams like the above, that
a consequence of our requirement of commutativity
is that wherever the diagram shows distinct generators of $A$
coming ``into'' and ``out of'' a vertex, e.g.,
\begin{picture}(18,16)
\put(5,15){\line(1,-1){10}}
\put(15,5){\line(-1,-1){10}}
\end{picture},
these must be part of a parallelogram
\begin{picture}(24,16)
\put(12,15){\line(1,-1){10}}
\put(22,5){\line(-1,-1){10}}
\put(12,15){\line(-1,-1){10}}
\put(2,5){\line(1,-1){10}}
\end{picture}.
So, for instance, if we tried to improve on~(\ref{d.3-gen<>}) by
deleting the line segment from $y_2$ to $z_2,$ so that~$b^2,$
though still nonzero because of its action on $x_1,$
ceased to equal $ac,$ thus increasing $\r{dim}_k\,A$
by $1,$ the resulting algebra would
not be commutative, because the configuration consisting
of $x_1,$ $y_1$ and $z_2$ and their connecting line segments
would not be part of a parallelogram.
(On the other hand, instances of
\begin{picture}(30,8)
\put(5,8){\line(1,-1){10}}
\put(15,-2){\line(1,1){10}}
\end{picture}
or of
\begin{picture}(30,8)
\put(5,-2){\line(1,1){10}}
\put(15,8){\line(1,-1){10}}
\end{picture}
do not need to belong to parallelograms, as illustrated
by $x_1,$ $x_2,$ $y_2$ in~\eqref{d.3-gen<>}.)

Our diagrammatic notation also allows us to
illustrate the fact mentioned in \S\ref{S.digr}, that a
faithful module over a finite-length homomorphic image
$A$ of $k[s,t]$ need not contain an isomorphic copy
of $A$ as a $\!k[s,t]\!$-module, though it will as a $\!k[s]\!$-module.
Let $A=k[s,t]/(s,t)^2,$ which has diagram
\begin{picture}(30,8)
\put(5,-2){\line(1,1){10}}
\put(15,8){\line(1,-1){10}}
\end{picture}.
Let $M$ be the dual module $\r{Hom}_k(A,k).$
This is faithful, but has diagram
\begin{picture}(30,8)
\put(5,8){\line(1,-1){10}}
\put(15,-2){\line(1,1){10}}
\end{picture},
which does not contain a copy of the diagram of $A.$
However, the diagrams for $A$ and $M$ as
$\!k[s]\!$-modules are
\begin{picture}(30,8)
\put(5,-2){\line(1,1){10}}
\put(25,-2){\circle*{1},}
\put(24.5,-1.5){\circle*{1},}
\end{picture}
and
\begin{picture}(30,8)
\put(5,8){\circle*{1}}
\put(5.5,7.5){\circle*{1}}
\put(15,-2){\line(1,1){10}}
\end{picture},
which are isomorphic.

\section{Some questions, and steps toward their answer}\label{S.genlz}

Theorem~\ref{T.Rt} has the unsatisfying feature that our $R$ has
absorbed one of the indeterminates
of the original algebra $k[s,\,t],$ but not the other.
We may ask, without referring to indeterminates,

\begin{question}\label{Q.S}
For which commutative rings $S$ does the statement,
\begin{equation}\begin{minipage}[c]{35pc}\label{d.Q.S}
For every $\!S\!$-module $M$ of finite length,
if we let $A=S/\r{Ann}_S\,M,$ then\\
$\lt_S(A)\leq \lt_S(M),$
\end{minipage}\end{equation}
hold?
\end{question}

Theorem~\ref{T.Rt} says roughly that the class of such rings
includes the rings $R[t]$ where $R$ is a principal ideal domain.
A plausible generalization would be that it contains all rings $S$ such
that every maximal ideal $\m$ of $S$ satisfies $\lt(\m/\m^2)\leq 2.$
If, in fact, Gerstenhaber's result turns out to go over to
$\!3\!$-generator algebras of commuting matrices,
we can hope that~(\ref{d.Q.S}) even holds for all $S$
whose maximal ideals satisfy $\lt(\m/\m^2)\leq 3.$

(There is a slight difficulty with
regarding~(\ref{d.Q.S}) as a generalization of the
property of Theorem~\ref{T.Rt}.
When $S=R[t],$ Theorem~\ref{T.Rt} concerns
$\!S\!$-modules of finite length over $R,$
while~(\ref{d.Q.S}) concerns $\!S\!$-modules of finite length over $S,$
and these are not always the same.
For instance, if $R$ is a discrete valuation ring with maximal
ideal $(p),$ then the $\!R[t]\!$-module $R[t]/(pt-1)$
has length~$1$ as an $\!R[t]\!$-module,
since the ring $R[t]/(pt-1)$ is a field, but has infinite
length as an $\!R\!$-module.
Since ``length over $R$'' has no meaning for a module
over a ring $S$ that is not assumed to be built from a subring $R,$
we shall take condition~(\ref{d.Q.S}) as our focus from here on.)

Note that a commutative ring $S$
satisfies~(\ref{d.Q.S}) if and only if all of its
finite-length homomorphic images $A$ do; equivalently,
if and only if all those images have the
stated property for {\em faithful} $\!A\!$-modules $M.$
Now for a ring, being of finite length is equivalent to being Artinian,
and every commutative Artinian ring is a finite
direct product of local rings.
This leads to the modified question,

\begin{question}\label{Q.A}
For which commutative Artinian local rings $A$ does the statement
\begin{equation}\begin{minipage}[c]{35pc}\label{d.Q.A}
Every faithful $\!A\!$-module $M$ satisfies $\lt_A(M)\geq\lt_A(A)$
\end{minipage}\end{equation}
hold?
\end{question}

We can get further mileage on these questions by combining
Theorem~\ref{T.Rt} with some theorems of I.\,S.\,Cohen~\cite{ISC}.
(Note to the reader of that paper: a
``local ring'' there means what is now called a Noetherian local ring.
Since the local rings we apply Cohen's results to will
be Artinian, this will be no problem to us.
Incidentally, Cohen defines a {\em generalized} local ring to mean
what we would call a (not necessarily Noetherian)
local ring whose maximal ideal $\m$ is
finitely generated and satisfies $\bigcap\m^n=\{0\},$
and he comments that he does not know whether every such ring is
``local'', i.e., is also Noetherian.
This has been answered in the negative~\cite{WH+MR}.)

Recall that a local ring $A$ with maximal ideal $\m$
is said to be {\em equicharacteristic}
if the characteristics of $A$ and $A/\m$ are the same.
This is equivalent to saying that $A$ contains a field.
(The implication from ``contains a field''
to ``equicharacteristic'' is clear.
Conversely, note that since $A/\m$ is a field, its characteristic
is $0$ or a prime number $p.$
In the former case, every member of $\Z-\{0\}$ is invertible
in $A/\m,$ and hence in $A,$ so $A$ contains the field $\Q;$
while in the latter, if $A$ is equicharacteristic, then,
like $A/\m,$ it has characteristic $p,$
and so contains the field $\Z/(p).)$

Cohen shows in~\cite[Theorem~9, p.\,72]{ISC} that a complete
Noetherian local ring which is
equicharacteristic is a homomorphic image of the ring of formal
power series in $\lt(\m/\m^2)$ indeterminates over a field,
where $\m$ is the maximal ideal of the ring.
Using this, we can get

\begin{proposition}\label{P.equichar}
Suppose $A$ is a commutative local Artinian ring with maximal
ideal $\m,$ and that $\lt(\m/\m^2)\leq 2.$
Then if $A$ is equicharacteristic, it satisfies~\textup{(\ref{d.Q.A})}.

Hence, if $S$ is a commutative ring such that
every maximal ideal $\m\subseteq S$ satisfies $\lt(\m/\m^2)\leq 2,$
and $S$ contains a field, then $S$ satisfies \textup{(\ref{d.Q.S})}.
\end{proposition}

\begin{proof}
We shall prove the first assertion.
Clearly, the second will then
follow by applying the first to local factor-rings of $S.$

Since the local ring $A$ is Artinian, it is complete,
so by the result of Cohen's cited, it
is a homomorphic image of a formal power series ring
in $\leq 2$ indeterminates over a field.
But a finite-length homomorphic image of a formal power series
ring is an image of the corresponding polynomial ring.
Hence by the result of Gerstenhaber with which we started,
$A$ satisfies~(\ref{d.Q.A}).
\end{proof}

Cohen's result for mixed characteristic is
\cite[Theorem~12, p.\,84]{ISC}.
The case we shall use, that of the last sentence of that theorem,
says that if $A$ is a complete Noetherian local ring
whose residue field $A/\m$ has characteristic $p$
(i.e., such that $p\in\m),$ but such that $p\notin\m^2,$ then $A$
can be written as a homomorphic image of a formal power series
ring in $\lt(\m/\m^2)-1$ indeterminates over a complete discrete
valuation ring $V$ in which $p$ has valuation~$1.$
(Intuitively, $p$ takes the place of one of the
indeterminates in the result for the equicharacteristic case.)
This gives us

\begin{proposition}\label{P.mixed}
Again let $A$ be a commutative local Artinian ring with maximal
ideal $\m,$ such that $\lt(\m/\m^2)\leq 2.$
If $p\in\m$ but
$p\notin\m^2,$ then $A$ satisfies~\textup{(\ref{d.Q.A})}.

Hence, if $S$ is a commutative ring such that every maximal ideal
$\m\subseteq S$ satisfies $\lt(\m/\m^2)\leq 2,$
and such that no prime $p\in\Z$ belongs to the square
of any maximal ideal of $S,$ then $S$ satisfies \textup{(\ref{d.Q.S})}.
\end{proposition}

\begin{proof}
We will prove the first assertion.
The second will then follow by applying that assertion to
local factor rings whose residue fields
have prime characteristic, while applying the first
assertion of Proposition~\ref{P.equichar} to local factor rings
whose residue fields have characteristic zero.

In the situation of the first assertion,
the result of Cohen cited, again combined with the observation
that a finite-length homomorphic image of a formal power series
ring is a homomorphic image of the corresponding
polynomial ring, tells us that $A$ is a homomorphic image
of a polynomial ring
in at most one indeterminate over a discrete valuation ring $V.$
Hence by Theorem~\ref{T.Rt} above, $A$ satisfies~(\ref{d.Q.A}).
\end{proof}

If, in the mixed-characteristic case, we instead have $p\in\m^2,$
Cohen's result only
tells us that $A$ is a homomorphic image of a formal power
series ring in $\lt(\m/\m^2)$ (rather than
$\lt(\m/\m^2)-\nolinebreak 1)$ indeterminates over a complete
discrete valuation ring $V;$
so in our case, $A$ is a homomorphic image of $V[[s,\,t]].$
In general, this is not enough to give us the conclusion
we want, but there are cases where it is.
Let $d$ be the integer such that $p\in\m^d-\m^{d+1},$
and suppose that
\begin{equation}\begin{minipage}[c]{35pc}\label{d.p=q^d}
$p$ has a $\!d\!$-th root $q$ in $A.$
\end{minipage}\end{equation}
Then this $\!d\!$-th root $q$ will lie in $\m-\m^2,$
so via a change of variables, the indeterminate
$s$ in $V[[s,\,t]]$ may be taken to be an element
that maps to $q\in A.$
Thus, $A$ is a homomorphic image of $V[[s,\,t]]/(s^d-p);$
so using, as before, the fact that $A$ has finite length, we see that
$A$ is in fact a homomorphic image of $V[s,\,t]/(s^d-p).$
But $V[s]/(s^d-p)$ is a discrete valuation ring $V'\supseteq V,$
so $A$ is a homomorphic image of $V'[t],$ and we can again
conclude from Theorem~\ref{T.Rt} that it satisfies~(\ref{d.Q.A}).

Can we generalize this further?
It might seem harmless to weaken~(\ref{d.p=q^d})
to say that some {\em associate} of $p$ in $A$
has a $\!d\!$-th root $q\in A.$
But then the problem arises of
where the unit of $A$ that carries $p$ to $q^d$ lies.
If it does not belong to the image of $V,$ we can't use
it in constructing our extension $V'.$
We might hope to incorporate the condition
that {}that unit lie in the image of $V$ into
a generalization of condition~(\ref{d.p=q^d}); but a
version of Proposition~\ref{P.mixed} based on such a condition
would be awkward to formulate, since the
$V$ given by Cohen's result is not part of the hypothesis of
Proposition~\ref{P.mixed}.
One assumption that {\em will} clearly guarantee that we can
argue as suggested is that the unit in question lie in the
image of $\Z$ in $A.$
I will not try here to find the ``best'' result of this sort.

If we don't assume any condition like~(\ref{d.p=q^d}),
there are examples where $A$ indeed
cannot be generated by one element over
a homomorphic image of a discrete valuation ring.
For instance, let $p$ be any prime, and within
$\Z\,[\,p^{1/5}],$ let us take the subring
$\Z\,[\,p^{2/5},\,p^{3/5}]$ and divide out by the ideal $(p^2),$
writing
\begin{equation}\begin{minipage}[c]{35pc}\label{d.2/5,3/5}
$A\ =\ \Z\,[\,p^{2/5},\,p^{3/5}]/(p^2).$
\end{minipage}\end{equation}

We see that $A$ is local and Artinian, with
maximal ideal $\m$ generated by $\{\,p^{2/5},\,p^{3/5},\,p\};$
and since the last of these elements is the product of the
first two, $\m$ is in fact $\!2\!$-generated, and $\m/\m^2$
can be seen to have length~$2.$
But I claim that $A$ is not $\!1\!$-generated over
a homomorphic image $B$ of a valuation ring $V.$
Roughly speaking, if it were, then that
subring $B\subseteq A$ would either have to have
the property that all its elements
are associates of powers of $p^{2/5},$
or that they are associates of powers of $p^{3/5};$
but $p\in B$ cannot be either.

Nevertheless, I would be surprised if the ring~(\ref{d.2/5,3/5})
did not satisfy~(\ref{d.Q.A}).
Any way I can think of to construct a candidate counterexample
could be duplicated over $k[s^2,\,s^3]/(s^{10})$ for $k$ a field,
though we know that no counterexample exists in that case
by Gerstenhaber's original result.

We can in fact show that for {\em all but finitely many} primes~$p,$
the ring~(\ref{d.2/5,3/5}) does satisfy~(\ref{d.Q.A}).
For suppose we had counterexamples for an infinite set $P$ of primes.
Let us write $A_p$ $(\,p\in P)$ for the corresponding
rings~(\ref{d.2/5,3/5}), and choose for each $p\in P$ an
$\!A_p\!$-module $M_p$ witnessing the failure of~(\ref{d.Q.A}).
Now let $A$ be an ultraproduct of the $A_p$ with
respect to a nonprincipal ultrafilter on $P,$ and $M$
the corresponding ultraproduct of the $M_p,$ an $\!A\!$-module.
From the fact that the $A_p$ all have the same length (namely $10),$
one can verify that $A$ will also have that length, hence be Artinian,
and from the fact that $\lt\m/\m^2=2$ for all $A_p,$ one finds
that the same is true for $A.$
Moreover, the characteristic of $A/\m$ will be $0,$ because every
prime integer is invertible in all but one of the $A_p/\m_p;$
hence $A$ is necessarily equicharacteristic.
The ultraproduct $M$ will be a faithful $\!A\!$-module, and
since by assumption all the $M_p$ have lengths less than the common
length of the $A_p,$ the module $M$ will also have
length less than that common value.
Hence $M$ witnesses the failure of~(\ref{d.Q.A}) for $A,$
contradicting Proposition~\ref{P.equichar};
so there cannot be such an infinite set $P$ of primes.

We see that the above method of reasoning in fact gives

\begin{proposition}\label{P.ultra}
For every positive integer $n,$ there are at most finitely
many primes $p$ for which there exist commutative local
Artinian rings $A$ of length $n$ and characteristic
a power of $p$ which satisfy $\lt(\m/\m^2)\leq 2,$ but fail to
satisfy\textup{~(\ref{d.Q.A}).}\qed
\end{proposition}

Above, we have, for brevity, been focusing on the more
challenging aspects of our problem.
One can also formally extend our results in more trivial ways.
For instance, using the case of Cohen's \cite[Theorem~12, p.\,84]{ISC}
that does not make the assumption $p\notin\m^2$
(quoted following Proposition~\ref{P.mixed}),
we see that any $A$ having $\lt(\m/\m^2)\leq 1$
satisfies~(\ref{d.Q.A}), with no need for a condition on the
behavior of integer primes $p.$
Also, one can easily extend the
final assertion of Proposition~\ref{P.equichar}
to a commutative ring $S$ which, rather than containing a
field, contains a direct product of fields, or more
generally, a von Neumann regular subring.
Still more generally, using the first statements of both those
propositions, we can extend the second statements
thereof to rings $S$ such
that for every maximal ideal $\m\subseteq S$ and prime $p\in\Z,$
either $p\notin\m^2$ or $p\in\bigcap_n\m^n.$

Let us now turn to rings $S$ and $A$ for which we can show
that~(\ref{d.Q.S}) or~(\ref{d.Q.A}) does {\em not} hold.
The first assertion of the next result generalizes our observations on
the algebra described in~(\ref{d.abxy}) and~(\ref{d.abcd}).
(This will be clearer from the proof than from the statement.)

\begin{proposition}\label{P.I_1,I_2}
Suppose $A$ is a commutative local Artinian ring, with maximal
ideal $\m.$
If $A$ has ideals $I_1$ and $I_2$ with zero intersection,
such that $A/I_1$ and $A/I_2$ have isomorphic submodules
$J_1/I_1\cong J_2/I_2$ satisfying
\begin{equation}\begin{minipage}[c]{35pc}\label{d.lt<}
$\lt(A/I_1)\,+\,\lt(A/I_2)\,-\,\lt(J_1/I_1)\,<\,\lt(A),$
\end{minipage}\end{equation}
\textup{(}equivalently, $\lt(A)<\lt(J_1)+\lt(I_2)),$
then $A$ does not satisfy~\textup{(\ref{d.Q.A}).}

In particular, this is the case if $A$ is any commutative local Artinian
ring satisfying $\m^2=\{0\}$ and $\lt(\m)\geq\nolinebreak 4.$

Hence, no commutative ring $S$ having a maximal ideal
$\m$ with $\lt(\m/\m^2)\geq 4$ satisfies\textup{~(\ref{d.Q.S}).}
\textup{(}In this last statement,
we do not require $\lt(\m/\m^2)$ to be finite.\textup{)}
\end{proposition}

\begin{proof}
In the situation of the first paragraph, let $M$ be the
module obtained from $A/I_1\oplus A/I_2$ by identifying
the isomorphic submodules
$J_1/I_1\subseteq A/I_1$ and $J_2/I_2\subseteq A/I_2.$
Each of $A/I_1$ and $A/I_2$ still
embeds in $M,$ so since the annihilators
$I_1$ and $I_2$ of these modules
have zero intersection, $M$ is a faithful $\!A\!$-module.
Now $\lt(M)$ is given by the left-hand side of~(\ref{d.lt<}),
hence that inequality shows the failure of~(\ref{d.Q.A}).
The parenthetical statement of equivalence on the
line after~(\ref{d.lt<}) is seen by expanding
the expressions of the form ``$\lt(P/Q)$'' in~(\ref{d.lt<})
as $\lt(P)-\lt(Q),$ and simplifying.

(In the example described in~(\ref{d.abxy}) and~(\ref{d.abcd}),
we can take $I_1=\r{Ann}_A(x) = (c,\,d),$
$I_2=\r{Ann}_A(y) = (a^m,\,a^{m-1}b,\dots,b^m),$
$J_1=\{f\in A\mid fx\in Acy+Ady\}=I_1+(a^m,b^m),$ and
$J_2=\{f\in A\mid fy\in Acy+Ady\}=I_2+(c,d).)$

To get the assertion of the second paragraph, let $\lt(\m)=d,$
so that $\m$ can be regarded as a $\!d\!$-dimensional
vector space over $A/\m.$
Let $I_1$ and $I_2$ be any subspaces of $\m$
of equal dimension $e\geq 2,$ and having
zero intersection (these exist because
$d\geq 4),$ and let $J_1=J_2=\m.$
By comparison of dimensions, $J_1/I_1\cong J_2/I_2$
as $\!A/\m\!$-modules, and hence as $\!A\!$-modules.
Now $\lt(J_1)+\lt(I_2) =d+e>d+1=\lt(A),$ giving the inequality
noted parenthetically as equivalent to~(\ref{d.lt<}).

For $S$ and $\m$ as in the final statement, let $A_0$
be the local ring $S/\m^2,$ with square-zero
maximal ideal $\m_0=\m/\m^2.$
Since $A_0$ need not have finite length, let us divide
out by an $\!A_0/\m_0\!$-subspace of $\m_0$
whose codimension is finite but $\geq 4.$
The result is a homomorphic image $A$ of $S$ which has
finite length and, by the second assertion
of the lemma, fails to satisfy~(\ref{d.Q.A}).
Hence $S$ fails to satisfy~(\ref{d.Q.S}).
\end{proof}

If it should turn out that~(\ref{d.Q.A}) holds for
every $A$ with $\lt(\m/\m^2)\leq 3,$ we would have a complete answer
to Question~\ref{Q.S}; for comparing that fact
with Proposition~\ref{P.I_1,I_2},
we could conclude that the rings $S$ satisfying~(\ref{d.Q.S})
are precisely those for which all maximal ideals $\m$
satisfy $\lt(\m/\m^2)\leq\nolinebreak 3.$

From the proof of Proposition~\ref{P.I_1,I_2},
we can see that the existence of ideals satisfying~(\ref{d.lt<})
is necessary and sufficient for
the existence of a {\em $\!2\!$-generator} $\!A\!$-module $M$
witnessing the failure of~(\ref{d.Q.A}).
For higher numbers of generators, it seems hard to formulate
similar necessary and sufficient conditions;
though one can give {\em sufficient} conditions,
corresponding to necessary and sufficient conditions
for the existence of such modules with particular sorts
of structures (e.g., sums of three cyclic submodules, each pair
of which is glued together along a pair of isomorphic submodules),
and these might be useful in looking for examples.

We have seen that the algebras described in \S\ref{S.4-gen} are cases
of Proposition~\ref{P.I_1,I_2}.
For a further example, suppose we adjoin to $\Z$
the $\!7\!$-th root of a prime $p,$ and then pass to the subring
\begin{equation}\begin{minipage}[c]{35pc}\label{d.4567}
$S\ =\ \Z\,[\,p^{4/7},\,p^{5/7},\,p^{6/7}].$
\end{minipage}\end{equation}
This has a maximal ideal $\m$ generated
by $\{p^{4/7},\,p^{5/7},\,p^{6/7},\,p\},$
and these generators are linearly independent modulo $\m^2,$
so by Proposition~\ref{P.I_1,I_2}, $S$ does not satisfy~(\ref{d.Q.S}).

The next result will give us a further class of rings $A$
that {\em do} satisfy~(\ref{d.Q.A}).
However, this class is not closed under
homomorphic images, and can fail to satisfy~(\ref{d.Q.S}).
Thus, though the result will add to what we know regarding
Question~\ref{Q.A}, it says little about
Question~\ref{Q.S}, which inspired that question.

We recall that a commutative local Artinian
ring $A$ is said to be {\em Frobenius} if it is cocyclic as an
$\!A\!$-module, i.e., if its socle has length~$1.$
(For an Artinian but not-necessarily-commutative, not-necessarily-local
ring, the Frobenius condition is the statement
that the socle is isomorphic as right and as left module to $A/J(A)$
\cite[Theorem~16.14(4)]{TYL}.)

\begin{lemma}\label{L.cocyc}
Every Frobenius commutative local Artinian ring $A$
satisfies~\textup{(\ref{d.Q.A})}.
\end{lemma}

\begin{proof}
If $M$ is a faithful $\!A\!$-module, then $M$ has an
element $x$ not annihilated by $\r{socle}(A).$
Since $\r{socle}(A)$ is simple, the annihilator of $x$
has trivial intersection with that socle, hence is zero.
So $A\,x$ is a faithful cyclic $\!A\!$-module, hence
has length equal to the length of $A,$ so $\lt M\geq\lt A.$
\end{proof}

For an example of
a ring as in the above lemma which has $\lt(\m/\m^2)\geq 4,$
and therefore, though we have just seen
that it satisfies~(\ref{d.Q.A}), will not
satisfy~(\ref{d.Q.S}), let $k$ be a field, take any $n_1,\dots,n_4>0,$
and let $A=k[t_1,\,t_2,\,t_3,\,t_4]/
(t_1^{n_1+1},\,t_2^{n_2+1},\,t_3^{n_3+1},\,t_4^{n_4+1}).$
Then the socle of $A$ is the
$\!1\!$-dimensional space spanned by the element
$t_1^{n_1}\,t_2^{n_2}\,t_3^{n_3}\,t_4^{n_4},$ so $A$ is Frobenius,
but $\m/\m^2$ is $\!4\!$-dimensional, with basis
$t_1,\,t_2,\,t_3,\,t_4.$

Let us also note, in contrast with the above lemma,
that a {\em large} socle does not {\em prevent}
a ring from satisfying~(\ref{d.Q.A}).
For instance, for $k$ a field and $n$ any positive integer, the algebra
$A=k[s,\,t]/(s^n,\,s^{n-1}t,\dots,\linebreak[1]\,s\,t^{n-1},\,t^n)$
has socle of length $n,$ with basis
$\{s^{n-1},\,s^{n-2}\,t,\dots,\,t^{n-1}\},$ but
by Proposition~\ref{P.equichar}, $A$ satisfies~(\ref{d.Q.A}).

However, Lemma~\ref{L.cocyc}, together with
the idea of Propositions~\ref{P.equichar} and~\ref{P.mixed}, suggests

\begin{question}\label{Q.co2gen}
Does every commutative local Artinian ring $A$ whose socle has
length $\leq 2$ \textup{(}or even $\leq 3)$
satisfy~\textup{(\ref{d.Q.A})}?
\end{question}

In an appendix,~\S\ref{S.gens.socs}, we shall obtain
some results on modules over a not necessarily commutative ring $A,$
which, for $A$ commutative Artinian, generalize
Lemma~\ref{L.cocyc} to show that if $A$
has socle of length $n$ and does not satisfy\textup{~(\ref{d.Q.A})},
then any minimal-length $\!A\!$-module $M$ witnessing this
failure must be generated by $\leq n$ elements, and, dually,
must have socle of length $\leq n.$

We know from the module diagrammed in~(\ref{d.4-gen}) that
an $M$ generated by $2$ elements and also having $\r{socle}(M)$
of length $2$ can witness the failure of\textup{~(\ref{d.Q.A})}.
But note that in that example, $\r{socle}(A)$ has length~$4,$ and
the construction ``economizes'', using few vertices at the top
and bottom to host a large number of edges representing
elements of $\r{socle}(A)$ in between.
It seems likely that this is an instance of some general properties of
modules witnessing the failure of\textup{~(\ref{d.Q.A})}.
If so, it may
be possible to strengthen, for such modules, the bounds just
mentioned, as suggested in

\begin{question}\label{Q.better}
Let $A$ be a commutative local Artinian ring, with socle of length $n,$
and maximal ideal $\m.$

If $A$ does not satisfy\textup{~(\ref{d.Q.A})},
and $M$ is an $\!A\!$-module witnessing this fact, must
$\lt(M/\m M)$ and/or $\lt(\r{socle}(M))$ be $\leq n-1$?

In the above situation, and perhaps, more generally, if
$M$ is a faithful $\!A\!$-module satisfying $\lt(M)\leq\lt(A),$
must $\lt(M/\m M)+\lt(\r{socle}(M))\leq n+1$?
\end{question}

The final part of the above question is suggested by the observation
that faithful cyclic modules, faithful cocyclic modules, and
all the modules described in \S\ref{S.4-gen} and
\S\ref{S.3-gen} satisfy the stated inequality.
A positive answer to either part of the question would
immediately give a positive
answer to the ``length~$\leq 2$'' case of Question~\ref{Q.co2gen}.

The final result of this section concerns rings $A$ that are very small.
We first note

\begin{lemma}\label{L.lt3}
Any module $M$ of length $\leq 3$ over a
\textup{(}not necessarily commutative or Artinian\textup{)} ring $A$
is either a direct
sum of cyclic modules, or a direct sum of cocyclic modules.
\end{lemma}

\begin{proof}
If $M$ is not itself cocyclic, this means $\lt(\r{socle}(M)) \geq 2,$
and dually, if $M$ is not cyclic, then $\lt(M/\m M)\geq 2,$ which,
subtracting from $3,$ gives $\lt(\m M)\leq 1<\lt(\r{socle}(M)).$
Hence $\r{socle}(M)$ must have a simple submodule $N$
not contained in $\m M.$
If we take the inverse image $L\subseteq M$ of a complement
of the image of $N$ in the semisimple module $M/\m M,$
this will still not contain $N,$ and since $N$ has
length $1,$ we see that $M=N\oplus L.$
Now $\lt(L)\leq 2,$ and it is easy to see that a module
of that length is either simultaneously cyclic and
cocyclic, or a direct sum of two simple submodules, in either
case giving the desired decomposition of~$M.$
\end{proof}

\begin{corollary}\label{C.ltA_<3}
Suppose $A$ is a commutative local Artinian ring
which has length $\leq 4$ \textup{(}equivalently, whose maximal ideal
$\m$ has length $\leq 3).$
Then $A$ satisfies~\textup{(\ref{d.Q.A}).}
\end{corollary}

\begin{proof}
Any $M$ witnessing the failure
of~(\ref{d.Q.A}) would have length $<\lt(A),$ so by the above lemma,
it would be a direct sum of cyclic or of cocyclic modules.
We saw in \S\ref{S.3-gen} that a direct sum of {\em two} cyclic
or cocyclic modules cannot give a counterexample to~(\ref{d.Q.A}).
One can generalize this result to any number of cyclic or cocyclic
modules (using repeatedly the observation
$\lt(A/I)+\lt(A/J)>\lt(A/I\cap J)),$ giving the desired result.
\end{proof}

So, for instance, this result applies to
the ring of~(\ref{d.3-gen1}) (essentially, the ring
generated by matrix units $e_{13},$ $e_{14},$ $e_{24}).$
Another example can be gotten by taking the ring
$\Z\,[\,p^{2/5},\,p^{3/5}]/(p^2),$ for which I noted above that
I have not been able to prove~(\ref{d.Q.A}), and truncating it
a bit further, to $\Z\,[\,p^{2/5},\,p^{3/5}]/(p^{4/5},\,p^{6/5}).$
This has maximal ideal of length~$3,$ spanned modulo
its square by $\{\,p^{2/5},\,p^{3/5}\},$
and with square spanned by $p=p^{2/5}\,p^{3/5}.$
So this truncation does satisfy~(\ref{d.Q.A}).
(We remark that, the explicit rational powers of $p$ in our
description of this ring are illusory;
it could equally well be written $\Z/(p^2)[s,\,t]/(s^2,\,t^2,\,st-p).)$

\section{Other sorts of questions}\label{S.other_Qs}
Though the focus of this note has been the
condition $\lt A\leq \lt M,$ one can ask, more generally, how
big $\lt A/\lt M$ can become in cases where it may exceed~$1.$
To maximize the hope of positive
results, I will pose the question here for algebras of endomorphisms
of vector spaces.

\begin{question}\label{Q.r_d}
For each positive integer $d,$ let $r_d$ be the supremum
of the ratio $\r{dim}_k\,A/\r{dim}_k V,$ over all
commutative {\em $\!d\!$-generator} algebras $A$ of endomorphisms of
nonzero finite-dimensional vector-spaces $V$ over arbitrary fields $k.$

\textup{(}Thus, the $r_d$ form a nondecreasing sequence,
whose terms are real numbers or $+\infty.$
We know that $r_1=r_2=1,$ and the examples of~{\rm \S\ref{S.3-gen}}
and~{\rm \S\ref{S.4-gen}} respectively suggest
that $r_3$ may be $1,$ and $r_4$ may be $5/4.)$

Determine as much as possible about this sequence.
In particular, are all its terms finite?
\end{question}

I don't even see how to prove $r_3$ finite!
(If we did not restrict ourselves to {\em commuting} endomorphisms,
these suprema would become infinite for $d=2,$
since the full $n\times n$
matrix algebra can be generated by two matrices, and the ratio of its
dimension to that of the space on which it acts, $n^2/n=n,$
is unbounded.)

Something we {\em can} say is that as a function of $d,$
the $r_d$ increase without bound:

\begin{lemma}\label{L.r_d}
If $d,\,e_0,\,e_1$ are positive integers such that $d\geq e_0\,e_1,$
then $r_d\geq (e_0\,e_1+1)/(e_0+e_1).$

Hence, taking $e_0=2m-1,$ $e_1=2m+1,$ we see that $r_{4m^2-1}\geq m.$
\end{lemma}

\begin{proof}
Let $V$ be the direct sum of an $\!e_0\!$-dimensional vector space $V_0$
and an $\!e_1\!$-dimensional space $V_1,$
and $A$ be the algebra of endomorphisms of $V$ spanned by
the identity, and all endomorphisms that carry $V_0$ into $V_1$
and annihilate $V_1.$
Since any two endomorphisms of the latter sort have product $0,$
$A$ is commutative.
It is generated as an algebra by any basis of the
$\!e_0 e_1\!$-dimensional space of
such endomorphisms, hence
{\em a fortiori} it can be generated by $d\geq e_0\,e_1$ elements.
Since $A$ has dimension $e_0\,e_1+1$ and $V$ has dimension
$e_0+e_1,$ we get $r_d\geq (e_0\,e_1+1)/(e_0+e_1),$ as claimed.

The final sentence clearly follows.
\end{proof}

In the spirit of \cite[\S3]{E}, we might expect that the
inequalities we have obtained for algebras of endomorphisms
of vector spaces would entail analogous inequalities for
monoids of endomaps of sets, with cardinalities replacing dimensions.
But this is not the case.
For instance, a $\!1\!$-generator group of permutations of an
$\!n\!$-element set can have order much larger than $n,$ if the
generating permutation has many cycles of relatively prime lengths.
(The reason why results on algebras don't imply the
corresponding results for monoids is that the matrices corresponding
to a family of distinct
endomaps of a finite set need not be linearly independent.)

I will end by repeating, in slightly generalized form, a
question I asked in~\cite{E}, which resembles the subject considered
{\em here} (and differs from the subject considered {\em there})
in that it asks whether the size of a certain
family of actions is bounded by the size of the object it acts on; but
which is otherwise only loosely related to the topic of either paper.

\begin{question}[after {\cite[Question 23]{E}}]\label{Q.E}
Let $R$ be a commutative algebra over a commutative ring $k,$
let $V$ be a $\!k\!$-submodule of $R,$ and let $n$ be a
positive integer such that the $\!k\!$-submodule $V^n\subseteq R$
of all sums of $\!n\!$-fold products of
elements of $V$ has finite length as a $\!k\!$-module.
Then must
\begin{equation}\begin{minipage}[c]{35pc}\label{d.E}
$\lt_k\,(V/\r{Ann}_V(V^n))~\leq~\lt_k(V^n),$
\end{minipage}\end{equation}
where $\r{Ann}_V(V^n)$ denotes $\{x\in V~|~x\,V^n=\{0\}\,\}\subseteq V$?
\end{question}

\section{Acknowledgements}\label{S.ackn}
I am indebted to Arthur Ogus for asking whether two commuting
$n\times n$ matrices generate an algebra of dimension $\leq n$
(when neither of us was aware that this was a known result), and for
a suggestion he made in the ensuing discussion, which turned into the
proof of Corollary~\ref{C.CH}; and to Kevin O'Meara for much
helpful correspondence on these matters.

\section{Appendix: Some submodules and factor modules}\label{S.gens.socs}

This section, except for the final corollary,
can be read independently of the rest of this note.
Rings are not here assumed commutative
(but are still associative and unital), and
``module'' means left module.
The Jacobson radical of a ring $A$ is denoted $J(A).$

Statement~(ii) in each of the next two results describes how
certain modules can be decomposed into fairly ``small'' modules:
in the first case, as a sum of submodules $N$ such
that $N/J(A)N$ is simple; in the second, as a subdirect
product of modules $L$ with $\r{socle}(L)$ simple.

\begin{lemma}\label{L.M_as_sum}
Let $A$ be an Artinian ring, and $M$ an $\!A\!$-module
\textup{(}not necessarily Artinian\textup{)}.
Then

\textup{(i)} If $S$ is a simple submodule of $M/J(A)M,$
then $M$ has a submodule $N$ such that the inclusion $N\subseteq M$
induces an isomorphism $N/J(A)N\cong S\subseteq M/J(A)M.$

Hence

\textup{(ii)} Given a decomposition of $M/J(A)M$
as a sum of simple modules $\sum_{i\in I} S_i,$ one
can write $M$ as the sum of a family of submodules
$N_i$ $(i\in I),$ such that for each $i,$ $N_i/J(A)N_i\cong S_i,$
and $S_i$ is the image of $N_i$ in $M/J(A)M.$
\end{lemma}

\begin{proof}
In the situation of~(i),
let $x$ be any element of $M$ whose image in $M/J(A)M$
is a nonzero member of $S.$
Thus, the image of $A\,x$ in $M/J(A)M$ is $S.$
Since $A$ is Artinian, $A\,x$ has finite length, hence we can
find a submodule $N\subseteq A\,x$ minimal for having $S$ as its image
in $M/J(A)M.$
Now since $N/J(A)N$ is semisimple, it
has a submodule $S'$ which maps isomorphically to $S$ in $M/J(A)M.$
If $S'$ were a proper submodule of $N/J(A)N,$ then its inverse
image in $N$ would
be a proper submodule of $N$ which still mapped surjectively
to $S,$ contradicting the minimality of $N.$
Hence $N/J(A)N=S'\cong S,$ completing the proof of~(i).

In the situation of~(ii), choose for each $S_i$
a submodule $N_i\subseteq M$ as in~(i).
Then we see that $M=J(A)M+\sum_i N_i,$ hence since $A$ is Artinian,
$M=\sum_i N_i$ \cite[Theorem~23.16\,$(1)\!{\implies}\!(2')$]{TYL1},
as required.
\end{proof}

The next result is of a dual sort, but the arguments can be
carried out in a much more general context, so that the result we are
aiming for (the final sentence) looks like an afterthought.
Note that submodules called $S$ etc.\ are not here assumed simple.

\begin{lemma}\label{L.M_as_subdir}
Let $A$ be a ring and $M$ an $\!A\!$-module.
Then

\textup{(i)} If $S$ is any submodule of $M,$ then $M$
has a homomorphic image $M/N$ such that the composite
map $S\hookrightarrow M\to M/N$ is an embedding, and
the embedded image of $S$ is essential in $M/N.$

Hence

\textup{(ii)} If $E$ is an essential submodule of $M,$
and $f: E\to\prod_I S_i$ a subdirect decomposition of $E,$
then there exists a subdirect decomposition $g:M\to\prod_i M_i$
of $M,$ such that each $M_i$ is an overmodule of $S_i$ in which $S_i$
is essential, and $f$ is the restriction of $g$ to $E\subseteq M.$

In particular, every locally Artinian module can be written
as a subdirect product of locally Artinian modules with simple socles.
\end{lemma}

\begin{proof}
In the situation of~(i), let $N$ be maximal among submodules
of $M$ having trivial intersection with $S.$
The triviality of this intersection means that $S$ embeds
in $M/N,$ while the maximality condition makes
the image of $S$ essential therein.
(Indeed, if it were not essential, $M/N$ would have a
nonzero submodule $T$ disjoint from the image of $S,$
and the inverse image of $T$ in $M$ would contradict the
maximality of $N.)$

In the situation of~(ii), for each $j\in I$ let $K_j$ be the kernel of
the composite $E\to\prod_I S_i\to S_j.$
Applying statement~(i) with $M/K_j$ in the role of $M,$ and
$E/K_j\cong S_j$ in the role of $S,$ we get an
image $M_j$ of $M/K_j,$ and
hence of $M,$ in which $S_j$ is embedded and is essential.
Now since $E$ is essential in $M,$ every
nonzero submodule $T\subseteq M$ has
nonzero intersection with $E,$ and that
intersection has nonzero projection to $S_i$ for some $i;$
so in particular, for that $i,$ $T$ has nonzero image in $M_i.$
Since this is true for every $T,$
the map $M\to\prod_I M_i$ is one-to-one,
and gives the desired subdirect decomposition.

The final assertion follows from the fact that the socle
of a locally Artinian module is essential, and, being semisimple,
can be written as a subdirect product (indeed, as a direct sum)
of simple modules.
\end{proof}

We can now get the following result, showing that given a
{\em faithful} module $M$ over an Artinian ring, we can carve
out of $M$ a ``small'' faithful submodule,
factor-module, or subfactor.
Note that in the statement, though length has its usual meaning for
modules, the length of $\r{socle}(A)$ as a bimodule may be
less than its length as a left (or right) module.

\begin{proposition}\label{P.socle}
Let $A$ be an Artinian ring,
let $n$ be the length of $\r{socle}(A)$ as a bimodule
\textup{(}equivalently, as a $\!2\!$-sided ideal\textup{)},
and let $M$ be a faithful $\!A\!$-module.
Then

\textup{(i)} $M$ has a submodule $M'$ which is again faithful over $A,$
and satisfies $\lt(M'/J(A)\,M')\leq n.$
\textup{(}In particular, $M'$ is generated by $\leq n$
elements.\textup{)}

\textup{(ii)} $M$ has a homomorphic image $M''$ which is faithful
over $A,$ and satisfies $\lt(\r{socle}(M''))\leq\nolinebreak n.$

\textup{(iii)} $M$ has a subfactor faithful
over $A$ which has both these properties.
\end{proposition}

\begin{proof}
To get~(i), write $M/J(A)M$ as a direct sum of simple
$\!A\!$-modules $S_i$ $(i\in I),$ and take a generating family
of submodules $N_i\subseteq M$ as in Lemma~\ref{L.M_as_sum}(ii).
Since $M=\sum_i N_i$ is faithful, and $\r{socle}(A)$
has length $\leq n$ as a
$\!2\!$-sided ideal, the sum of some family of $\leq n$
of these submodules, say
\begin{equation}\begin{minipage}[c]{35pc}\label{d.M'}
$M'\ =\ N_{i_1}+\dots+N_{i_m}$\quad where $m\leq n,$
\end{minipage}\end{equation}
must have the property that $M'$ is annihilated by no
nonzero subideal of $\r{socle}(A).$
(Details: one chooses the $N_{i_j}$ recursively.
As long as $N_{i_1}+\dots+N_{i_j}$ is annihilated by a nonzero
subideal $I\subseteq\r{socle}(A),$ one can choose an $N_{i_{j+1}}$
which fails to be annihilated by some member of $I.$
Thus, the annihilators in $\r{socle}(A)$ of successive
sums $M'=N_{i_1}+\dots+N_{i_j}$ $(j=0,1,\dots)$
form a strictly decreasing chain.
By our assumption on the length of $\r{socle}(A),$ this chain
must terminate after $\leq n$ steps
with a sum $M'$ annihilated by no nonzero elements
of $\r{socle}(A).)$
But an ideal of an Artinian ring having zero intersection
with the socle is zero, so $M'$ has zero annihilator, i.e., is faithful.
Since each $N_i$ satisfies $\lt(N_i/J(A)N_i)=1,$
we have $\lt(M'/J(A)M')\leq m\leq n.$

Statement (ii) is proved in the analogous way
from the final statement of Lemma~\ref{L.M_as_subdir}, using
images of $M$ in products of finite subfamilies
of the $M_i$ in place of submodules of $M$
generated by finite subfamilies of the $N_i.$

Statement (iii) follows by combining~(i) and~(ii).
\end{proof}

Returning to the ideas of the body of this note, we deduce from
Proposition~\ref{P.socle}(iii) the following
fact which was mentioned before Question~\ref{Q.better}.

\begin{corollary}\label{C.socle+Q.A}
If a commutative local Artinian ring $A$ having socle of
length $n$ fails to satisfy\textup{~(\ref{d.Q.A})}, then any $M$
of minimal length among $\!A\!$-modules witnessing this failure is
generated by $\leq n$ elements and has socle of length $\leq n.$
\end{corollary}

\begin{proof}
By the minimal-length assumption on $M,$ no proper subfactor
of $M$ can be faithful, hence the subfactor given by
Proposition~\ref{P.socle}(iii) must be $M$ itself.
\end{proof}

\end{document}